\documentclass{amsart}
\usepackage{amssymb}

\usepackage[all]{xy}

\newtheorem{theorem}{Theorem}[section]
\newtheorem{claim}[theorem]{Claim}

\newtheorem{corollary}[theorem]{Corollary}
\newtheorem{lemma}[theorem]{Lemma}

\theoremstyle{definition}

\theoremstyle{definition}
\newtheorem{remark}[theorem]{Remark}

\newcommand{\IN}{\mathbb N}
\newcommand{\IR}{\mathbb R}

\newcommand{\IZ}{\mathbb Z}
\newcommand{\IP}{\mathbb P}

\newcommand{\e}{\varepsilon}
\newcommand{\w}{\omega}

\newcommand{\Ra}{\Rightarrow}
\newcommand{\Tau}{\mathcal T}

\title[A quantitative generalization of Prodanov-Stoyanov Theorem]{A quantitative generalization of Prodanov-Stoyanov Theorem\\ on minimal Abelian topological groups}
\author{Taras Banakh}
\address{T.~Banakh: Ivan Franko National University of Lviv
(Ukraine) and\newline\indent Jan Kochanowski University in Kielce
(Poland)}
\email{t.o.banakh@gmail.com}
\dedicatory{To the memory of Ivan Prodanov (1935--1985)}

\begin{document}
\begin{abstract} A topological group $X$ is defined to have {\em compact exponent} if for some number $n\in\IN$ the set $\{x^n:x\in X\}$ has compact closure in $X$. Any such number $n$ will be called a {\em compact exponent} of $X$. Our principal result states that a complete Abelian topological group $X$ has compact exponent (equal to $n\in\IN$) if and only if for any injective continuous homomorphism $f:X\to Y$ to a topological group $Y$ and every $y\in \overline{f(X)}$ there exists a positive number $k$ (equal to $n$) such that $y^k\in f(X)$. This result has many interesting implications: (1) an Abelian topological group is compact if and only if it is complete in each weaker Hausdorff group topology;
(2) each minimal Abelian topological group is precompact (this is the famous Prodanov-Stoyanov Theorem); (3) a topological group $X$ is complete and has compact exponent if and only if it is closed in each
Hausdorff paratopological group containing $X$ as a topological subgroup (this confirms an old conjecture of Banakh and Ravsky).
\end{abstract}
\maketitle

\section{Introduction}

In this paper we prove a theorem, which can be considered as a quantitative version of the fundamental Prodanov-Stoyanov Theorem on the precompactness of minimal Abelian topological groups. All topological groups considered in this paper are Hausdorff.

We recall that a topological group $X$ is {\em precompact} if its completion in the two-sided uniformity is compact. This happens if and only if $X$ is {\em totally bounded} in the sense that for any neighborhood $U\subset X$ of the unit there exists a finite subset $F\subset X$ such that $X=FU$. For a subset $A$ of a topological space $X$ by $\bar A$ we denote the closure of $A$ in $X$.

We shall say that a topological group $X$ has {\em compact exponent} if for some number $n\in\IN$ the set $nX=\{x^n:x\in X\}$ has compact closure $\overline{nX}$ in $X$. In this case, the number $n$ is called  {\em a compact exponent} of $X$. Observe that a topological group is compact if and only if 1 is its compact exponent.  For the first time the concept of compact exponent (without naming) has appeared in Banakh and Ravsky  \cite{BaR2001}.

By a {\em powertopological semigroup} we understand a semigroup $S$ endowed with a Hausdorff topology such that for any $a,b\in S$ and $n\in\IN$ the map $S\to S$, $x\mapsto ax^nb$, is continuous. It is clear that each topological semigroup is powertopological and each powertopological semigroup is {\em semitopological} which means that for any $a,b\in S$ the shift $S\to S$, $x\mapsto axb$, is continuous.

This paper contains a unique principal result:

\begin{theorem}\label{t:main} Let $n\in\IN$ be any number. For a complete Abelian topological group $X$ the following conditions are equivalent:
\begin{enumerate}
\item $X$ has compact exponent (equal to $n$);
\item for any injective continuous homomorphism $f:X\to Y$ to a topological group $Y$ and every point $y\in \overline{f(X)}\subset Y$ there exists a number $k\in\IN$ (equal to $n$) such that $y^k\in f(X)$;
\item for any continuous homomorphism $f:X\to Y$ to a powertopological semigroup $Y$ and every point $y\in \overline{f(X)}\subset Y$ there exists a number $k\in\IN$ (equal to $n$) such that $y^k\in f(X)$.
\end{enumerate}
\end{theorem}

This theorem has many interesting and non-trivial implications.

\begin{corollary}\label{c1} An Abelian topological group $X$ is compact if and only if $X$ for any injective continuous homomorphism $f:X\to Y$ to a topological group $Y$ the image $f(X)$ is closed in $Y$.
\end{corollary}

This corollary can be reformulated as follows.

\begin{corollary}\label{c2} An Abelian topological group $X$ is compact if and only if $X$ is complete in each weaker Hausdorff group topology.
\end{corollary}

This corollary implies the famous Prodanov-Stoyanov Theorem \cite{PS} on the precompactness of minimal Abelian topological groups. We recall \cite{D98}, \cite{DM} that a topological group $X$ is {\em minimal} if $X$ does not admit a strictly weaker Hausdorff group topology.

\begin{corollary}[Prodanov, Stoyanov]\label{c:PS} Each minimal Abelian topological group $X$ is precompact.
\end{corollary}

\begin{proof} The precompactness of $X$ is equivalent to the compactness of the completion $\bar X$ of $X$. By Corollary~\ref{c1}, the compactness of $\bar X$ will follow as soon as we check that for any continuous injective homomorphism $f:\bar X\to Y$ to a topological group $Y$ the image $f(\bar X)$ is closed in $Y$. We lose no generality assuming that the topological group $Y$ is complete and $f(\bar X)$ is dense in $Y$. By the minimality of $X$, the restriction $f|X:X\to f(X)$ is a topological isomorphism. Then it uniquely extends to a topological isomorphism of the completions $\bar X$ and $\overline{f(X)}=Y$ of the topological groups $X$ and $f(X)$. So, $f:\bar X\to Y$ is a topological isomorphism and $f(\bar X)=Y$ is closed in $Y$.
\end{proof}

Another corollary of Theorem~\ref{t:main} confirms an old conjecture of Banakh and Ravsky \cite{BaR2001}, \cite{Ravsky2003}.

\begin{corollary} For a complete Abelian topological group $X$ the following conditions are equivalent:
\begin{enumerate}
\item $X$ has compact exponent;
\item for any continuous homomorphism $f:X\to Y$ to a Hausdorff powertopological semigroup $Y$ the image $f(X)$ is closed in $Y$;
\item for any injective continuous homomorphism $f:X\to Y$ to a topological group $Y$ the quotient group $\overline{f(X)}/f(X)$ is a torsion group;
\item for any isomorphic topological embedding $f:X\to Y$ of $X$ into a Hausdorff paratopological group $Y$ the image $f(X)$ is closed in $Y$.
\end{enumerate}
\end{corollary}

\begin{proof} The equivalences $(1)\Leftrightarrow(2)\Leftrightarrow(3)$ follow immediately from the corresponding equivalences in Theorem~\ref{t:main} and $(3)\Leftrightarrow(4)$ was proved by Ravsky \cite{Ravsky2003}.
\end{proof}

\begin{remark} As was noticed by Dikranjan and Megrelishvili \cite[\S5]{DM}, the Stoyanov-Prodanov Theorem~\ref{c:PS} fails for nilpotent groups: the Weyl-Heisenberg group $H(w_0)=H(\IR)/Z$ where $$H(\IR)=\left\{\left(\begin{array}{ccc}1&a&b\\0&1&c\\0&0&1\end{array}\right):a,b,c\in\IR\right\}
\mbox{ \ \ and \ \ } Z=\left\{\left(\begin{array}{ccc}1&0&b\\0&1&0\\0&0&1\end{array}\right):b\in\IZ\right\}$$
is nilpotent Lie group (of nilpotence degree 2), which is minimal but not compact, see also \cite{DM2010} and \cite{Meg}.
Nonetheless, Dikranjan and Uspenskij \cite{DU} proved the following two extensions of Corollary~\ref{c1} to nilpotent and solvable topological groups.
\end{remark}

\begin{theorem}[Dikranjan, Uspenskij] A nilpotent topological group $X$ is compact if and only if for every continuous homomorphism $f:X\to Y$ to a topological group $Y$ the image $f(X)$ is closed in $Y$.
\end{theorem}

\begin{theorem}[Dikranjan, Uspenskij] A solvable topological group $X$ is compact if and only if for every continuous homomorphism $f:Z\to Y$ from a closed normal subgroup $Z\subset X$ of $X$ to a topological group $Y$ the image $f(Z)$ is closed in $Y$.
\end{theorem}

Our proof of Theorem~\ref{t:main} follows the line of the proof of Stoyanov-Prodanov Theorem from \cite{DPS}, with an additional care of (non)completeness and compact exponent. Because of that it is quite long and technical. We separate the proof into several steps. In Section~\ref{s:prel} we establish some preliminary results related to bounded sets in topological groups and properties of topological groups of (pre)compact exponent. In Section~\ref{s:Key}, on each Abelian topological group $X$ we define a weaker group topology $\tau_\diamond$ dependent on a sequence $(x_n)_{n\in\w}$ of points in $X$ and prove two Key Lemmas. The first Key Lemma~\ref{l:seq} produces a sequence $(x_n)_{n\in\w}$ in $X$ with some additional properties and the second Key Lemma~\ref{l:Hausdorff} shows that these properties ensure that the topology $\tau_\diamond$ is Hausdorff.
We shall apply these Key Lemmas once in the proof of the ``bounded'' version of Theorem~\ref{t:main} (with $k=n$) and three times for the ``unbounded'' version (with $k$ arbitrary). So, Theorem~\ref{t:main} follows from its ``bounded'' and ``unbounded'' versions, proved in Theorem~\ref{t:bound} and \ref{t:unbound}, respectively.

\section{Preliminaries}\label{s:prel}

In this section we fix some standard notations used in the paper. Also we recall some known results and prove some simple lemmas. %All topological semigroups considered in this paper are Hausdorff.

\subsection{Standard Notations}

By $\IN$ we denote the set of natural numbers, i.e., positive integer numbers and by $\w$ the set of non-negative integer numbers (= finite ordinals).  By $\IZ$ we denote the additive group of integer numbers.

By $\IP$ we denote the set of all prime numbers. For a number $n\in\IN$ let $\IP_n$ be the set of prime divisors of $n$.

For a subset $A$ of a topological space $X$ by $\bar A$ we denote the closure of $A$ in $X$. A point $x\in X$ is called an {\em accumulation point} of a sequence $(x_n)_{n\in\w}$ in a topological space $X$ if each neighborhood $O_x\subset X$ of $x$ contains infinitely many points $x_n$, $n\in\w$.

For group $X$ by $e$ we denote its unit, i.e., a unique element $e$ such that $ex=x=xe$. For two subsets $A,B$ of a group $X$ we put $AB=\{ab:a\in A,\;b\in B\}$. The powers $A^n$ are defined by induction: $A^0=\{e\}$ and $A^{n+1}=A^nA$ for $n\in\w$. Also $A^{-1}=\{a^{-1}:a\in A\}$ and $A^{-n}=(A^n)^{-1}=(A^{-1})^n$ for $n\in\w$.

For a sequence $(A_n)_{n\in\w}$ of subsets of a group $X$ by $\sum_{n\in\w}A_n$ we denote the subgroup of $X$ generated by the union $\bigcup_{n\in\w}A_n$.

For a subset $A$ of a group $X$ and a natural number $n\in\IN$ let $nA=\{a^n:a\in A\}$ be the set of $n$-th powers of the elements of $A$.
Observe that $nA\subset A^n$ and $nA\ne A^n$ in general. If $A$ is a subgroup of an Abelian group $X$, then $nA$ is a subgroup of $X$.

A topological group $X$ is {\em complete} if it is complete in its two-sided uniformity (generated by the entourages $\{(x,y)\in X\times X:y\in xU\cap Ux\}$ where $U$ runs over neighborhoods of the unit $e$) of $X$).

It is well-known (see \cite[\S3.6]{AT} or \cite{RD}) that the completion $\bar X$ of a topological group $X$ by its two-sided uniformity carries a natural structure of the topological group, which contains $X$ as a dense subgroup.

%A subset $Z$ of a group $X$ is called {\em central} if $zx=xz$ for any $z\in Z$ and $x\in X$.

\subsection{Bounded sets in topological groups}

In this section we establish some properties of bounded and precompact sets in topological groups.

A subset $B$ of a group $X$ is called {\em $U$-bounded} for a set $U\subset X$ if $B\subset FU\cap UF$ for some finite subset $F\subset X$.

A subset $X$ of a topological group is called {\em totally bounded} if it is $U$-bounded for any neighborhood $U$ of the unit in $X$. It is well-known that a subset $B$ of a complete topological group $X$ has compact closure $\bar B$ in $X$ if and only if $B$ is totally bounded. Because of that totally bounded subsets in topological groups are also called {\em precompact sets}, see \cite[\S3.7]{AT}.

The following simple (but important) lemma shows that the
failure of the total boundedness can be recognized by countable subgroups.

\begin{lemma}\label{l:separ} If a subset $A$ of a group $X$ is not $UU^{-1}$-bounded for some set $U\subset X$, then for some countable subgroup $Z\subset X$ the set $Z\cap A$ is not $U$-bounded.
\end{lemma}

\begin{proof} Since $A$ is not $UU^{-1}$-bounded, for every finite set $F=F^{-1}\subset X$ we have $A\not\subset FUU^{-1}\cap UU^{-1}F$ and hence $A^{-1}\not\subset UU^{-1}F^{-1}\cup F^{-1}UU^{-1}$. By Zorn's Lemma, there exists a maximal subset $E\subset A\cup A^{-1}$ which is {\em $UU^{-1}$-separated} in the sense that $x\notin yUU^{-1}$ for any distinct points $x,y\in E$. The maximality of $E$ guarantees that for any $x\in A\cup A^{-1}$ there exists $y\in E$ such that $x\in yUU^{-1}$ or $y\in xUU^{-1}$ (and hence $x\in yUU^{-1}$). Consequently, $A\cup A^{-1}=EUU^{-1}$ and $A\subset EUU^{-1}\cap UU^{-1}E^{-1}$. The choice of $U$ ensures that the set $E$ is infinite. Then we can choose any countable subgroup $Z\subset X$ that has infinite intersection $Z\cap E$ with $E$.

We claim that the set $Z\cap A$ is not $U$-bounded. Assuming the opposite, we could find a finite subset $F\subset X$ such that $Z\cap A\subset UF$. Since the set $Z\cap E$ is infinite, for some $x\in F$ there are two distinct points $y,z\in Z\cap E$ such that $y^n,z^n\in Ux$. Then $z^n\in Ux\subset UU^{-1}y^n$, which contradict the choice of the set $E$.
\end{proof}

\subsection{Localizations}

%For an element $x$ of a (topological) group $X$ by $\IZ(x)$ we denote the cyclic subgroup $\{x^n:n\in\IZ\}$ generated by $x$ (and by $\bar\IZ(x)$ the closure of $\IZ(x)$ in $X$).

%For a topological group $X$ by $T(X)$ we denote the subset of $X$ consisting of the elements $x\in X$ generating totally bounded cyclic subgroup $\IZ(x)$. Also for a prime number $p$ let $T_p(X)$ be the set of points $x\in T(X)$ such that for any $n\in\IN$, not divisible by $p$, the set $n\IZ(x)$ is dense in $\IZ(x)$.

A topological group $X$ is called {\em $p$-singular} for a prime number $p$ if for any $n\in\IN$, not divisible by $p$ the set $nX$ is dense in $X$.
For an Abelian topological group $X$ the union $T_p(X)$ of all $p$-singular subgroups of $X$ is the largest (closed) $p$-singular subgroup of $X$, see \cite[\S2.6]{DPS}.
For an Abelian discrete topological group $X$ the set $T_p(X)$ coincides with the $p$-Sylow subgroup $\{x\in X:\exists n\in\IN\;\;x^{p^n}=e\}$.
%If $X$ is compact, then the subgroup $\prod_{p\in\IP}T_p(X)$ generated by the union $\bigcup_{p\in \IP}T_p(X)$ is dense in $X$. Thi

We shall need the following fact whose proof can be found in \cite[2.6.2]{DPS}.

\begin{lemma}\label{l:Tp-dense} If an Abelian topological group is a union of compact subgroups, then the subgroup $\sum_{p\in\IP}T_p(X)$ is dense in $X$.
\end{lemma}

\begin{lemma}\label{l:Tp-comp} If $X$ is a compact Abelian topological group, then for any prime number $p$ and any number $n\in\IN$ we have $nT_p(X)=p^{n_p}T_p(X)$, where $n_p\in\w$ is the largest number such that $p^{n_p}$ divides $n$.
\end{lemma}

\begin{lemma}\label{l:Tp-dense2} Let $X$ be an Abelian topological group. For any number $n\in\IN$ we have $$\sum_{p\in\IP}p^{n_p}T_p(X)\subset \overline{nX},$$
where $n_p\in\w$ is the largest number such that $p^{n_p}$ divides $n$.
Moreover, if $X$ is a union of compact subgroups, then  $\sum_{p\in\IP}p^{n_p}T_p(X)$ is dense in $\overline{nX}$.
\end{lemma}

\begin{proof} For every $p$ the $p$-singularity of $T_p(X)$ implies that $\overline{p^{n_p}T_p(X)}=\overline{nT_p(X)}\subset \overline{nX}$. Taking into account that $\overline{nX}$ is a group, we conclude that
$\sum_{p\in\IP}p^{n_p}T_p(X)\subset \overline{nX}.$

Now assume that $X$ is a union of compact subgroups. To see that $\sum_{p\in\IP}p^{n_p}T_p(X)$ in dense in $\overline{nX}$, it suffices for any open set $U\subset X$ intersecting $nX$ to find an element $x\in U\cap \overline{nX}\cap\sum_{p\in\IP}p^{n_p}T_p(X)$. Since $U\cap nX\ne \emptyset$, there exists a point $z\in X$ such that $z^n\in U$. Choose a neighborhood $V\subset X$ of $z$ such that $V^n\subset U$. Since $X$ is a union of compact subgroups, the closure $\bar Z$ of the cyclic group $Z:=\{z^n\}_{n\in\IZ}$ is compact. By Lemma~\ref{l:Tp-dense}, the subgroup $\sum_{p\in\IP}T_p(\bar Z)$ is dense in $\bar Z$. So, we can find an element $v\in V\cap\sum_{p\in\IP}T_p(\bar Z)$. Lemma~\ref{l:Tp-comp} implies that $nT_p(\bar Z)=p^{n_p}T_p(\bar Z)$. Then $v^n\in V^n\cap \sum_{p\in\IP}nT_p(\bar Z)=V^n\cap\sum_{p\in\IP}p^{n_p}T_p(\bar Z)\subset U\cap\sum_{p\in\IP}p^{n_p}T_p(X)$.
\end{proof}

\subsection{The Bohr topology and F\o lner Theorem}

The {\em Bohr topology} on a semitopological group $X$ is the weakest topology on $X$ in which all continuous homomorphisms into compact topological groups remain continuous. The Bohr topology is not necessarily Hausdorff. Topological groups with Hausdorff Bohr topology and called {\em maximally almost periodic}.

In the proof of Theorem~\ref{t:main} we shall apply
the following fundamental theorem of F\o lner \cite{Folner}, whose proof can be found in \cite[1.4.3]{DPS}.

\begin{theorem}[F\o lner, 1954]\label{t:Folner} Let $X$ be an Abelian semitopological group such that for every $n\in\IZ$ the map $X\to X$, $x\mapsto x^n$, is continuous. If the group $X$ is $E$-bounded for some subset $E\subset X$, then for any neighborhood $U\subset X$ the set $UU^{-1}(EE^{-1})^8$ is a neighborhood of the unit in the Bohr topology of $X$.
\end{theorem}

\section{Two Key Lemmas}\label{s:Key}

In this section we prove the two lemmas following the ideas of the proof of Lemma 2.7.3 from \cite{DPS}. The first lemma will help us to construct a nice sequence $(x_n)_{n\in\w}$ in an Abelian topological group $X$. This sequence determines a weaker group topology $\tau_\diamond$ on $X$ and in the second lemma we find conditions under which this weaker topology is Hausdorff.

\begin{lemma}\label{l:seq} Let $X$ be an Abelian group,  $(z_k)_{k\in\w}$ be a sequence in $X$, $(W_k)_{n\in\w}$ be a decreasing sequence of sets in $X$ and $(G_k)_{k\in\w}$ be a sequence of subgroups of $X$ satisfying the following conditions:
\begin{enumerate}
\item for every $k\in\IN$ the set $W_k$ contains the unit $e$ of $X$ and $W_k=W_k^{-1}$;
\item for any $n,k\in\IN$ either $nG_0\subset G_k$ or $nG_0$ is not $W_k^2G_k$-bounded.
\end{enumerate}
Then there exists a sequence of points $(x_m)_{m\in\w}$ in $G_0$ satisfying the conditions:
\begin{enumerate}
\item[(3)] for any numbers $m\in\IN$, $n\le 2^{m+1}$ and $k\le m$ with $nG_0\not\subset G_k$ we have $x_m^n\notin F_{m}W_kG_k$, where
\item[(4)] $F_m=\big\{z_i^{\e_0}x_1^{\e_1}\cdots x_{m-1}^{\e_{m-1}}:i\le m,\;\e_0,\e_1,\dots,\e_{m-1}\in\IZ,\;\max\limits_{0\le i<m}|\e_i|\le 2^{m+1}\big\}$.
\end{enumerate}
\end{lemma}

\begin{proof}
% Fix a Banach measure $\mu$ on the Abelian group $X$.
The construction of the sequence $(x_k)_{k\in\IN}$ is inductive. Let $x_0=e$ and assume that for some $m\in\IN$ the points $x_0,\dots,x_{m-1}$ have been constructed. Consider the finite set $F_m$ defined by the condition $(4)$ of Lemma~\ref{l:seq}. Let $P_m$ be the set of all pairs $(n,k)\in\IN\times\IN$ such that $k\le n\le 2^{m+1}$ and $nG_0\not\subset G_k$. By the condition (2), for every $(n,k)\in P_m$ the set $nG_0$ is not $W_{k}^2G_{k}$-bounded.

The group $G_0$, being Abelian, is {\em amenable} and hence admits an invariant finitely additive invariant probability measure $\mu:\mathcal P(G_0)\to[0,1]$ defined on the Boolean algebra of all subsets of $G_0$, see \cite[0.15]{Pater}.

\begin{claim}\label{cl:measure} For any $(n,k)\in P_m$ and any $z\in F_m$ the set $S_{n,k,z}=\{x\in G_0:x^n\in zW_{k}G_{k}\}$ has measure $\mu(S_{n,k,z})=0$.
\end{claim}

\begin{proof} Using Zorn's Lemma, choose a maximal subset $M\subset G_0$ such that $x^{n}\notin y^{n}W_{k}^2G_{k}$ for any distinct points $x,y\in M$. The maximality of $M$ guarantees that for every $x\in G_0$ there exists $y\in M$ such that $x^{n}\in y^{n}W_{k}^2G_{k}$ or $y^{n}\in x^{n}W^2_{k}G_{k}$ and hence $x^{n}\in y^{n}W^{-2}_{k}G_{k}=y^nW^2_{k}G_{k}$. This means that $nX\subset MW^2_{k}G_{k}$.
Since the set $nG_0$ is not $W^2_{k}G_{k}$-bounded, the set $M$ is infinite.

We claim that for any distinct points $x,y\in M$ the sets $x^{n}W_{k}G_{k}$ and $y^{n}W_{k}G_{k}$ are disjoint. Assuming that the intersection
$x^{n}W_{k}G_{k}\cap y^{n}W_{k}G_{k}$ contains some point
$z$, we conclude that $y^n\in zW_{k}^{-1}G_k^{-1}\subset x^nW_kG_kW_k^{-1}G_k^{-1}=x^nW_k^2G_k$, which contradicts the choice of the set $M$.

Therefore, the family $(x^{n}W_{k}G_{k})_{x\in M}$ is disjoint and so is the family $(x^{n}zW_{k}G_{k})_{x\in M}$. Now consider the homomorphism $p:G_0\to G_0$, $p:x\mapsto x^{n}$ and observe that the family $\big(p^{-1}(x^nzW_{k}G_{k})\big)_{x\in M}=\big(xp^{-1}(zW_{k}G_{k})\big)_{x\in M}$ is disjoint. Now the additivity and invariantness of the measure $\mu$ ensures that the set $p^{-1}(zW_{k}G_{k})=\{x\in X:x^n\in zW_{k}G_{k}\}=S_{n,k,z}$ has measure $\mu(S_{n,k,z})=0$.
\end{proof}

The additivity of the measure $\mu$ and Claim~\ref{cl:measure} imply that the set $S=\bigcup_{(n,k)\in P_m}\bigcup_{z\in F_m}S_{n,k,z}$ has measure zero. Consequently, we can find a point $x_m\in G_0\setminus S$. It is clear that this point $x_m$ satisfies the condition (3) of Lemma~\ref{l:seq}.
\end{proof}

Let $X$ be an Abelian topological group and $(x_n)_{n\in\IN}$ be a sequence of points in $X$. For every positive real number $\delta$ consider the set $$\diamondsuit_\delta=\Big\{x_1^{\e_1}\cdots x_m^{\e_m}:m\in\IN,\;\e_1,\dots,\e_m\in\IZ,\;\;\sum_{i=1}^m\frac{|\e_i|}{2^i}<\delta,\;\sum_{i=1}^n\e_i=0\Big\}\subset X.$$ Let $\Tau$ be the topology of $X$ and $\mathcal B$ be the largest precompact group topology on the group $X$, i.e., $\mathcal B$ is the Bohr topology of the group $X$, endowed with the discrete topology. Let $\Tau_e=\{T\in\Tau:e\in T=T^{-1}\}$ and $\mathcal B_e=\{B\in\mathcal B:e\in B=B^{-1}\}$ be the families of open symmetric neighborhoods of the unit $e$ in the topological groups $(X,\Tau)$ and $(X,\mathcal B)$, respectively.

It is easy to see that the family $$\tau_e:=\{T\cdot(\diamondsuit_\delta\cap B):T\in\Tau,\;\delta>0,\;B\in\mathcal B\}$$ is a neighborhood base at the unit for some group topology $\tau_\diamond$ on $X$.
This topology $\tau_\diamond$ will be called {\em the topology, $\diamondsuit$-generated by} the sequence $(x_n)_{n\in\w}$, and this sequence will be called {\em $\diamondsuit$-generating} the topology $\tau_\diamond$.

If the topology $\tau_\diamond$ is Hausdorff, then its $\diamondsuit$-generating sequence $(x_n)_{n\in\w}$ will be called {\em $\diamondsuit$-Hausdorff}. In this case $X_\diamond:=(X,\tau_\diamond)$ is a Hausdorff topological group and we can consider its completion $\bar X_\diamond$, which is a complete topological group.

The following lemma detects $\diamondsuit$-Hausdorff sequences.

\begin{lemma}\label{l:Hausdorff} Let $X$ be an Abelian topological group, $(z_n)_{n\in\w}$ be a sequence in $X$, $(W_k)_{n\in\w}$ be a decreasing sequence of neighborhoods of the unit $e$ in $X$ and $(G_k)_{k\in\w}$ be a sequence of closed subgroups of $X$ satisfying the following conditions:
\begin{enumerate}
\item for every $k\in\IN$ we have $W_k=W_k^{-1}$ and $W_{k+1}^2\subset W_k$;
\item for any $n\in\IN$ and $k\le n$ either $nG_0\subset G_k$ or $nG_0$ is not $W_k^2G_k$-bounded;
\item the closed subgroup $G_\w:=\bigcap_{k\in\w}G_k$ is compact;
\item for every neighborhood $U\subset X$ of the unit there exists $k\in\IN$ such that $G_k\subset G_\w U$.
\end{enumerate}
If a sequence $(x_m)_{m\in\w}$ of points of $G_0$ satisfies the condition \textup{(3)} of\/ {\rm Lemma~\ref{l:seq}}, then it is $\diamondsuit$-Hausdorff and the set $\{x_m\}_{m\in\w}$ is totally bounded in the topological group $(X,\tau_\diamond)$.
\end{lemma}

\begin{proof} We separate the proof of this lemma into four claims.

\begin{claim}\label{cl:H1} For every $k\in\w$ we have $\diamondsuit_1\cap W_{k} G_k\subset G_{k}$.
\end{claim}

\begin{proof} To derive a contradiction, assume that the set
$(\diamondsuit_1\cap W_k G_k)\setminus G_k$ contains some point $x$. Write $x$ as $x=x_1^{\e_1}\cdots x_l^{\e_l}$ for some number $l\in\IN$ and some integer numbers $\e_1,\dots,\e_l$ such that $\sum_{i=1}^l\frac{|\e_i|}{2^i}<1$.  Since $x\notin G_k$, for some $i\le l$, we have $x_i^{\e_i}\notin G_k$.
Let $j\le l$ be the largest number such that $x_j^{\e_j}\notin G_k$ and thus $\e_jG_0\not\subset G_k$. The maximality of $j$ guarantees that  $x_i^{\e_i}\in G_k$ for every $i\in(j,l]$, and hence
$$W_kG_{k}\ni x=x_1^{\e_1}\cdots x_{j}^{\e_j}G_k$$and
$x_j^{\e_j}\in x_1^{-\e_1}\cdots x_{j-1}^{-\e_{j-1}}W_kG_kG_k^{-1}\subset F_jW_kG_k$, which contradicts the property (3) of Lemma~\ref{l:seq}.
\end{proof}

\begin{claim}\label{cl:H2} For any $k\in\IN$ and any $x\in \diamondsuit_1\setminus G_k$ there exists $\e>0$ such that $x\notin\diamondsuit_\e W_{k}$.
\end{claim}

\begin{proof} Write $x$ as $x=x_1^{\delta_1}\cdots x_n^{\delta_n}$ for some $n\in\IN$ and some integer numbers $\delta_1,\dots,\delta_n$ such that $\sum_{i=1}^n\frac{|\delta_i|}{p^i}<1$. Since $x\notin G_k$, for some $i\le n$ we have $x_i^{\delta_i}\notin G_k$.

We claim that the real number $\e=\frac1{2^{n}}$ has the required property.
To derive a contradiction, assume that $x\in \diamondsuit_\e W_{k}$ and find a point $y\in\diamondsuit_\e$ such that $x\in yW_{k}$.
Write $y$ as $y=x_1^{\e_1}\cdots x_m^{\e_m}$ for some integer numbers $m>n$ and $\e_1,\dots,\e_m$ such that $\sum_{i=1}^m\frac{\e_i}{2^i}<\e=\frac1{2^{n}}$. The latter inequality implies that $\e_i=0$ for all $i\le n$.

The inclusion $x\in yW_{k}$ implies that $y^{-1}x\in W_{k}$ and $y^{-1}x=x_1^{\gamma_1}\cdots x_m^{\gamma_m}$ where $\gamma_i=\delta_i$ for $i\le n$ and $\gamma_i=-\e_i$ for $n<i\le m$.
It follows that $|\gamma_i|\le\max\{|\delta_i|,|\e_i|\}<2^i$.
Let $j\le m$ be the largest number such that $x_j^{\gamma_j}\notin G_k$. Such number $j$ exists since $x_i^{\gamma_i}=x_i^{\delta_i}\notin G_k$ for some $i\le n$. Then $\gamma_iG_0\not\subset G_k$.

The maximality of $j$ guarantees that $x_{i}^{\delta_i}\in G_k$ for all $i\in(j,m]$.
It follows that $W_{k}\ni y^{-1}x\in x_1^{\gamma_1}\cdots x_{j-1}^{\gamma_{j-1}}x_j^{\gamma_j}G_k$ and hence
$x_j^{\gamma_{j}}\in x_1^{-\gamma_1}\cdots x_{j-1}^{-\gamma_{j-1}}W_{k}G_k\subset F_jW_kG_k$, which contradicts the condition (3) of Lemma~\ref{l:seq}.
\end{proof}

\begin{claim} The topology $\tau_\diamond$ is Hausdorff.
\end{claim}

\begin{proof} Given an element $x\in X\setminus\{e\}$, we should find $T\in\Tau_e$, $\e>0$ and $B\in\mathcal B_e$ such that $x\notin T\cdot(\diamondsuit_\e\cap B)$. If $x$ does not belong to the closure $\bar H$ of the subgroup $H$, generated by the set $\{x_n\}_{n\in\IN}$, then we can find a neighborhood $T=T^{-1}\in\Tau_e$ of unit in $X$ such that $xT\cap H=\emptyset$ and conclude that $T\diamondsuit_1 \subset TH$ does not contain $x$.

Next, we consider the case $x\in \bar H\setminus G_\w=\bigcup_{k\in\w}\bar H\setminus G_k$ and hence $x\notin G_k$ for some $k\in\IN$. Since the subgroup $G_k$ is closed in $X$, there exists a neighborhood $T\in\Tau_e$ of the unit such that $T=T^{-1}\subset W_{k+1}$ and $xT\cap G_k=\emptyset$. If $x\notin T\diamondsuit_1$, then, we are done. So, assume that $x\in T\diamondsuit_1$. In this case there exists $z\in\diamondsuit_1$ such that $x\in Tz$. It follows that $z\in xT^{-1}\cap \diamondsuit_1=xT\cap\diamondsuit_1\subset \diamondsuit_1\setminus G_k$.  Applying Claim~\ref{cl:H2}, find $\e>0$ such that $z\notin \diamondsuit_\e W_{k}$. We claim that $x\notin\diamondsuit_\e W_{k+1}$. In the opposite case, $z\in xT^{-1}\subset xW_{k+1}\subset \diamondsuit_\e W_{k+1}^2\subset \diamondsuit_\e W_{k}$, which contradicts the choice of $\e$.

Finally, we consider the case $x\in G_\w$. Choose a neighborhood $U\in\Tau_e$ of the unit such that $U=U^{-1}$ and $x\notin U^{37}$. By the condition (4) of Lemma~\ref{l:Hausdorff}, there exist $k\in\IN$ such that $G_k\subset G_\w U$. By the compactness of $G_\w$, there exists a finite set $F\subset G_\w\subset G_k$ such that $G_\w\subset FU$. Then $G_k\subset FU^2$. Consider the neighborhood $V=G_k\cap U^2$ of $e$ in the group $G_k$ and observe that $G_k=FV$. By F\o lner Theorem \ref{t:Folner}, the set $(VV^{-1})^9$ is a neighborhood of the unit in the Bohr topology of the group $G_k$ endowed with the discrete topology. The Baer Theorem \cite[4.1.2]{Rob} on extensions of homomorphisms into divisible groups implies that the Bohr topology of the group $G_k$ coincides with the subspace topology inherited from the Bohr topology on the group $X$. Consequently, there exists a Bohr neighborhood $B\in\mathcal B_e$ or $e$ such that $B\cap  G_k\subset (VV^{-1})^9\subset U^{36}$.
Let $T=W_{k}\cap U$.
We claim that the neighborhood $T\cdot(\diamondsuit_1 \cap B)\in\tau_\diamond$ of $e$ does not contain the point $x$. Assuming that $x\in T\cdot(\diamondsuit_1\cap B)$, we can find points $t\in T$ and $b\in\diamondsuit_1\cap B$ with $x=tb$ and conclude that $b=xt^{-1}\in G_\w T^{-1}\subset G_\w W_k$ and $b\in\diamondsuit_1\cap G_\w W_k\subset G_k$ according to Claim~\ref{cl:H1}. Then $b\in B\cap G_k\subset U^{36}$ and $x\in Tb\subset UU^{36}=U^{37}$, which contradicts the choice of the neighborhood $U$.
\end{proof}

\begin{claim}\label{cl:tot-bound} The set $\{x_n\}_{n\in\w}$ is totally bounded in the topological group $(X,\tau_\diamond)$.
\end{claim}

\begin{proof} Given a neighborhood $U\in\tau_\diamond$ of $e$, we need to find a finite set $F\subset X$ such that $\{x_n\}_{n\in\w}\subset FU$. We can assume that the neighborhood $U$ is of basic form
$U=T(\diamondsuit_\e\cap B)$ for some $T=T^{-1}\in\Tau_e$, $\e>0$ and $B=B^{-1}\in\mathcal B_e$.

Let $F\subset \{x_n\}_{n\in\w}$ be a maximal set such that $x\notin yU$ for any distinct elements $x,y\in F$. The maximality of $F$ guarantees that for any point $x\in \{x_n\}_{n\in\w}$ there exists a point $y\in F$ such that $x\in yU$ or $y\in xU$ (and $x\in yU^{-1}=yU$). So, $\{x_n\}_{n\in\w}\subset FU$. We claim that the set $F$ is finite.
To derive a contradiction, assume that $F$ is infinite.

Choose a neighborhood $D=D^{-1}\in\mathcal B_e$ of $e$ such that $D^2\subset B$. The total boundedness of the topology $\mathcal B$ yields a finite set $E\subset X$ such that $X=ED$. By the Pigeonhole principle, there exists a point $z\in E$ such that the set $zD$ contains two distinct points $x_n,x_m\in E$ with $\frac1{2^n}+\frac1{2^m}<\e$. The last inequality ensures that $x_nx_m^{-1}\in\diamondsuit_\e$.
 On the other hand, $x_nx_m^{-1}\in (zD)(zD)^{-1}=DD^{-1}\subset B$ and hence $x_nx_m^{-1}\in\diamondsuit_\e\cap B\subset U$ and $x_n\in x_mU$, which contradicts the definition of the set $F$. This contradiction shows that the set $F$ is finite and the set $\{x_k\}_{k\in\w}\subset FU$ is totally bounded.
\end{proof}
\end{proof}

\section{Topological groups of compact exponent}

We recall that a topological group $X$ has {\em compact exponent} if for some $n\in\IN$ the set $nX=\{x^n:x\in X\}$ has compact closure in $X$.

\begin{lemma}\label{l:CE} If a complete Abelian topological group has compact exponent, then $X$ contains a $\diamondsuit$-Hausdorff sequence $(x_m)_{m\in\w}$ which has an accumulation point $x_\infty$ in the completion $\bar X_\diamond$ of the topological group $X_\diamond:=(X,\tau_\diamond)$ such that for every $n\in\IN$ the power $x^n_\infty$ belongs to $X_\diamond$ if and only if the set $nX$ is precompact.
\end{lemma}

\begin{proof} Let $s\in\IN$ be the smallest number such that $sX$ is precompact. Using Lemma~\ref{l:separ}, find a countable subgroup $Z=\{z_k\}_{k\in\w}\subset X$ such that for every $n<s$ the set $nZ$ is not precompact.

Put $G_0=\bar Z$ and $G_k=\overline{sZ}$ for all $k\in\IN$.

 For every prime number $p$ let $s_p\in\w$ be the largest number such that $p^{s_p}$ divides $s$. The compactness of exponent implies that $X$ is a union of compact subgroups. By Lemma~\ref{l:Tp-dense2}, the subgroup $\sum_{p\in\IP}p^{s_p}T_p(\bar Z)$ is dense in $\overline{sZ}$.

\begin{claim}\label{cl:tq} For any prime divisor $q$ of $s$ the subgroup $p^{s_q-1}T_q(\bar Z)$ is not precompact.
\end{claim}

\begin{proof} Consider the number $t=s/q$ and for every $p\in\IP$ let $t_p$ be the largest number such that $p^{t_p}$ divides $t$. Observe that $t_{q}=s_q-1$ and $t_p=s_p$ for $p\ne q$. By Lemma~\ref{l:Tp-dense2}, the subgroup $\sum_{p\in\IP}p^{t_p}T_p(\bar Z)$ is dense in $\overline{tZ}$.
Assuming that $q^{s_q-1}T_q(\bar Z)$ is precompact, we would conclude that the subgroup $\sum_{p\in\IP}p^{t_p}T_p(\bar Z)=q^{s_q-1}T_q(\bar Z)\cdot \sum_{p\in\IP}p^{s_p}T_p(\bar Z)\subset q^{s^q-1}T_q(\bar Z)\cdot \overline{sX}$ is precompact
and so is its closure $\overline{tZ}$. But this contradicts the minimality of $s$.
\end{proof}

Choose a neighborhood $W_0\subset X$ of the unit such that for every divisor $p$ of $s$ the (non-precompact) subgroup $p^{s_p-1}T_p(\bar Z)$ is not $W_0^2sZ$-bounded. Choose a decreasing sequence $(W_k)_{k\in\IN}$ of neighborhoods of the unit such that $W_k=W_k^{-1}$ and $W_{k+1}^2\subset W_k$.

 It is easy to see that the sequences $(W_k)_{k\in\w}$ and $(G_k)_{k\in\w}$ satisfy the conditions (1)--(2) of Lemma~\ref{l:seq} and (1)--(4) of Lemma~\ref{l:Hausdorff}. Applying these lemmas, we obtain a $\diamondsuit$-Hausdorff sequence $(x_m)_{m\in\w}$ in $G_0$ satisfying the conditions (3),(4) of Lemma~\ref{l:seq} and such that the set $\{x_n\}_{n\in\w}$ is totally bounded in the Hausdorff topological group $X_\diamond=(X,\tau_\diamond)$ and hence has an  accumulation point $x_\infty\in \bar X_\diamond$ in the completion $\bar X_\diamond$ of  $X_\diamond=(X,\tau_\diamond)$.

\begin{claim} For every $n\in\IN$ the power $x_\infty^n$ belongs to $X_\diamond$ if and only if the subgroup $nX$ is precompact.
\end{claim}

\begin{proof} If the set $nX$ is precompact, then its closure $\overline{nX}\subset X$ is compact and hence closed in $\bar X_\diamond$. Then the dense subset $\{x\in\bar X:x^n\in\overline{nX}\}\supset X_\diamond$ is closed in $\bar X$  and hence contains the point $x_\infty$. Consequently, $x_\infty^n\in X_\diamond$.

Next, we assume that the set $nX$ is not precompact. In this case we should prove that $x_\infty^n\notin X_\diamond$. To derive a contradiction, assume that $x_\infty^n\in X_\diamond$.
First we observe that the number $s$ does not divide $n$. Otherwise $nX\subset sX$ would be precompact. Consequently, for some prime number $q$ with $s_q>0$ the power $q^{s_q}$ does not divide $n$.

Since $x_\infty^n\in X_\diamond$, for every neighborhood $V\subset X$ of the unit the neighborhood $x_{\infty}^n\diamondsuit_1V\subset x_\infty^n G_0V\in\tau_\diamond$ intersects the set $\{x^n_k\}_{k\in\w}\subset G_0$, which implies that $x^n_\infty\in \bar G_0=\bar Z$. Then there exists a number $l\in \w$ such that $x_\infty^n\in z_lW_s$. Let $\e=\frac1{2^{l+s}}$ and observe that $z_l\diamondsuit_\e W_s\in \tau_\diamond$ is a neighborhood of $x^n_\infty$ in the topology $\tau_\diamond$. Since $x_\infty^n$ is an accumulation point of the sequence $(x^n_k)_{k\in\w}$, there exists a number $k>\max\{l,ns\}$ such that $x^n_k\in z_l\diamondsuit_\e W_s$ and hence $x^n_k\in z_lyW_s$ for some $y\in\diamondsuit_\e$. Write the point $y$ as $y=x_1^{\e_1}\cdots x_\lambda^{\e_\lambda}$ for some $\lambda>k$ and some integer numbers $\e_1,\dots,\e_\lambda$ such that $\sum_{i=1}^\lambda\e_i=0$ and
$\sum_{i=1}^\lambda\frac{|\e_i|}{2^i}<\e=\frac1{2^{l+s}}$. The last inequality implies that $\e_i=0$ for all $i\le l+s$ and $|\e_i|<2^{i-s}$ for all $i\le\lambda$.

It follows that $z_lW_s\ni y^{-1}x_k^n=x_1^{\delta_1}\cdots x_\lambda^{\delta_\lambda}$ where $\delta_k=n-\e_k$ and $\delta_i=-\e_i$ for all $i\in\{1,\dots,\lambda\}\setminus \{k\}$.
We claim that $|\delta_is|<2^{i+1}$ for all $i\le\lambda$. If $i\ne k$, then $|\delta_is|=|\e_i|s<2^{i-s}s<2^i$. For $i=k$ we have $$|\delta_i|s=|n-\e_k|s\le ns+|\e_k|s\le k+2^{k-s}s<2^{k+1}=2^{i+1}.$$

Since $\sum_{i=1}^\lambda\delta_i=n$ and $n$ is not divisible by $q^{s_q}$, for some $i\le\lambda$, the number $\delta_i$ is not divisible by $q^{s_q}$. Let $j\le \lambda$ be the largest number such that $\delta_j$ is not divisible by $q^{s_q}$.
Since $\delta_i=-\e_i=0$ for $i\le l$, the number $j$ is greater than $l$.

The maximality of $j$ guarantees that $x_i^{\delta_i}\in q^{s_q}G_0$ for all $i\in(j,\lambda]$ and hence
$$x_j^{\delta_j}\in x_1^{-\delta_1}\cdots x_{j-1}^{-\delta_{j-1}}z_lW_1x_{j+1}^{-\delta_{j+1}}\cdots x_\lambda^{-\delta_\lambda}\subset
x_1^{-\delta_1}\cdots x_{j-1}^{-\delta_{j-1}}z_lW_sq^{s_q}\bar Z.$$
Let $s'=s/q^{s_q}$ and consider the power
$$x_j^{\delta_js'}\in x_1^{-\delta_1s'}\cdots x_{j-1}^{-\delta_{j-1}s'}z_l^{s'}W_s^{s'}(s'p^{s_p}\bar Z)\subset F_jW_1(s\bar Z)\subset F_jW_1G_1,$$which contradicts the condition (3) of Lemma~\ref{l:seq} as
$\max_{i\le j}|\delta_i s'|\le \max_{i\le j}2^{i+1}=2^{j+1}$ and
$\delta_js'Z\not\subset G_1=\overline{sZ}$. To see the last  non-inclusion, for every prime number $p$ let $t_p$ be the largest number such that $p^{t_p}$ divides $\delta_js'$. It follows that $t_q<s_q$ and $t_p\ge s_p$ for $p\ne q$.
By Lemma~\ref{l:Tp-dense2}, the closure of the set $\delta_js'Z$ coincides with the closure of the set $\sum_{p\in\IP}p^{t_p}T_p(\bar Z)$, which is not precompact because it contains the set $q^{t_q}T_q(\bar Z)\supset q^{s_q-1}T_q(\bar Z)$ which is not precompact according to Claim~\ref{cl:tq}.
Then the set $\delta_js'Z$ is not precompact and hence cannot be contained in the precompact set $\overline{sZ}=G_1$.
\end{proof}
\end{proof}

Lemma~\ref{l:CE} implies the ``bounded'' (or rather ``quantitative'') part of Theorem~\ref{t:main}.

\begin{theorem}\label{t:bound} For a number $n\in\IN$ and a complete Abelian topological group $X$ of compact exponent the following conditions are equivalent:
\begin{enumerate}
\item the subgroup $nX$ is totally bounded;
\item for any closed subsemigroup $Z\subset X$, any continuous homomorphism $f:Z\to Y$ to a powertopological semigroup $Y$ and any $y\in \overline{f(Z)}\subset Y$ we have $y^n\in f(X)$;
\item for any injective continuous homomorphism $f:X\to Y$ to a topological group $Y$ and any $y\in \overline{f(X)}$ we have $y^n\in f(X)$.
\end{enumerate}
\end{theorem}

\begin{proof} To prove that $(1)\Ra(2)$, assume that the subgroup $nX$ is totally bounded and hence has compact closure $\overline{nX}$ in the complete topological group $X$. Fix a closed subsemigroup $Z\subset X$ and a  continuous homomorphism $f:Z\to Y$ to a powertopological semigroup $Y$.
Being compact, the set $Z\cap\overline{nX}$ has compact (and hence closed) image $f(Z\cap \overline{nX})$ in the powertopological semigroup $Y$. The continuity of the power map $\pi:Y\to Y$, $\pi:y\mapsto y^n$, implies that the set $F=\{y\in Y:y^n\in f(Z\cap\overline{nX})\}$ is closed in $Y$. Since $f(X)\subset F$, the closed set $F$ contains the closure $\overline{f(X)}$ of $f(X)$ in $Y$. Then any point $y\in\overline{f(X)}$ belongs to $F$ and hence has $y^n\in f(Z\cap \overline{nX})\subset Y$.
\smallskip

The implication $(2)\Ra(3)$ is trivial and $(3)\Ra(1)$ follows from Lemma~\ref{l:CE}.
\end{proof}

\begin{remark} Theorem~\ref{t:bound}(2) cannot be further extended to semitopological semigroups: just take any infinite discrete topological group $X$ of compact exponent and consider its one-point compactification $Y=X\cup\{\infty\}$. Extend the group operation of $X$ to a semigroup operation on $Y$ letting $y\infty=\infty=\infty y$ for all $y\in Y$. It is easy to check that $Y$ is a compact Hausdorff semitopological semigroup containing $X$ as a non-closed discrete subgroup.
\end{remark}

\section{Topological groups of hypocompact exponent}

We shall say that a topological group $X$ has {\em hypocompact exponent} if for any neighborhood $U\subset X$ of the unit there exist $n\in\IN$ and a finite subset $F\subset X$ such that $nX\subset FU$. It is clear that each topological group of precompact exponent has hypocompact exponent.

\begin{lemma}\label{l:fg} If a topological group $X$ has hypocompact exponent, then each finitely generated Abelian subgroup $A$ of $X$ is precompact.
\end{lemma}

\begin{proof} Given any neighborhood $U\subset X$ of the unit, we need to find a finite subset $F\subset X$ such that $A\subset FU$. Since $X$ has hypocompact exponent, there exist $n\in \IN$ and $F_1\subset X$ such that $nX\subset F_1U$. Since the group $A$ is Abelian and finitely generated, the quotient group $A/nA$ is finite. So, we can find a finite subset $F_2\subset A$ such that $A=F_2\cdot nA$. Then $A=F_2\cdot nA\subset F_2F_1U$ and hence the finite set $F=F_1F_2$ has the required property.
\end{proof}

\begin{lemma}\label{l:converge} If a complete Abelian topological group $X$ has hypocompact exponent, then the subgroup $\overline{\w X}:=\bigcap_{n\in\IN}\overline{nX}$ is compact and for every $U\subset X$ there exists $n\in\IN$ such that $\overline{nX}\subset \overline{\w X}\cdot U$.
\end{lemma}

\begin{proof} To see that the closed set $\overline{\w X}$ is compact, it suffices to check that it is totally bounded in the complete topological group $X$. Given any open neighborhood $U\subset X$ of the unit, use the hypocompactness of the exponent of $X$ and find $m\in\IN$ and a finite set $F\subset X$ such that $mX\subset FU$ and hence $\overline{\w X}\subset \overline{m X}\subset F\bar U$, which means that $\overline{\w X}$ is precompact and hence compact.
\smallskip

We claim that $nX\subset \overline{\w X}\cdot U$ for some $n\in\IN$. In the opposite case, for every $n\in\IN$ we could choose a point $x_n\in n! X\setminus \overline{\w X}U$. Taking into account that the topological group $X$ has hypocompact exponent, we see that the set $\{x_n\}_{n\in\w}$ is precompact in the complete topological group $X$ and hence the sequence $(x_n)_{n\in\w}$ has an accumulation point $x_\infty\in X\setminus \overline{\w X}U$. On the other hand, for every $m\ge n$ we have $x_m\in m!X\subset n!X$, which implies that $x_\infty\in\bigcap_{n\in\IN}\overline{n X}$ and this is a desired contradiction.
\end{proof}

%Next, we apply the Key Lemmas~\ref{l:seq} and \ref{l:Hausdorff} to the $p$-components of $\w$-narrow Abelian topological groups of hypocompact exponent.

We shall say that a subset $B$ of a topological group $X$ is {\em precompact modulo} a set $A\subset X$ if for any open neighborhood $U\subset X$ of the unit, the set $B$ is $AU$-bounded.

\begin{lemma}\label{l:modulo} Let $A,B$ be two subsets of an Abelian topological group. If $A$ is precompact modulo $B$, then for any $n\in\IN$ the set $nA=\{a^n:a\in A\}$ is precompact modulo $nB$.
\end{lemma}

\begin{proof} Given any neighborhood $U\subset X$ of the unit, find a neighborhood $V\subset X$ of the unit such that $nV\subset U$. Since $A$ is precompact modulo $B$, there exists a finite subset $F\subset X$ such that $A\subset FBV$. It follows that for every $a\in A$, there are points $f\in F$, $b\in B$ and $v\in V$ such that $a=fbv$. Then $a^n=f^nb^nv^n$ and hence $nA\subset (nF)(nB)(nV)\subset (nF)(nB)U$, which means that $nA$ is precompact modulo $nB$.
\end{proof}

 A topological group $X$ is {\em $\w$-narrow} if for any neighborhood $U\subset X$ of the unit there exists a countable subset $C\subset X$ such that $X=CU$. It is clear that each separable topological group is $\w$-narrow.

\begin{lemma}\label{l:p-local} Assume that an $\w$-narrow complete Abelian topological group $X$ has hypocompact exponent and for some prime number $p$ the $p$-component $T_p(X)$ of $X$ does not have compact exponent. Then $X$ contains a $\diamondsuit$-Hausdorff sequence $\{x_m\}_{m\in\w}\subset T_p(X)$ which has an accumulation point $x_\infty$ in $\bar X_\diamond$ such that $x^n_\infty\notin X_\diamond$ for all $n\in\IN$.
\end{lemma}

\begin{proof} First we prove the following statement.

\begin{claim}\label{cl:modulo-p} For every $k\in\w$ the set $p^k T_p(X)$ is not precompact modulo $p^{k+1}T_p(X)$.
\end{claim}

\begin{proof} To derive a contradiction, assume that for some $k\in\w$ the set $p^kT_p(X)$ is precompact modulo $p^{k+1}T_p(X)$. Lemma~\ref{l:modulo} implies that for every $n\ge k$ the set $p^nT_p(X)$ is precompact modulo $p^{n+1}T_p(X)$ and hence $p^kT_p(X)$ is precompact modulo $p^nT_p(X)$. We claim that $p^kT_p(X)$ is precompact. Indeed, for every neighborhood $U$ of the unit, by the hypocompactness of the exponent of $X$, there exist $n\in\IN$ and finite subset $F\subset X$ such that $nT_p(X)\subset FU$. Let $d\in\w$ be the largest number such that $p^d$ divides $n$. Since the subgroup $T_p(X)$ is $p$-singular, the subgroup $nT_p(X)$ is dense in $p^dT_p(X)$. Consequently, $$p^dT_p(X)\subset\overline{p^dT_p(X)}=\overline{nT_p(X)}\subset \overline{nX}\subset F\bar U.$$ Since $p^kT_p(X)$ is precompact modulo $p^nT_p(X)$, there exists a finite subset $E\subset X$ such that $p^kT_p(X)\subset (p^nT_p(X))EU\subset FUEU=(FE)U^2$, which means that $p^kT_p(X)$ is precompact and hence the closed subgroup $T_p(X)$ of the complete group $X$ has compact exponent. But this contradicts the assumption of Lemma~\ref{l:p-local}.
\end{proof}

Using Claim~\ref{cl:modulo-p}, choose a decreasing sequence $(W_k)_{k\in\w}$ of neighborhoods of the unit in $X$ such that for every $k\in\w$ the following conditions are satisfied:
\begin{enumerate}
\item $W_k=W_k^{-1}$ and $W_{k+1}^2\subset W_k$;
\item $p^kT_p(X)$ is not $W_{k}^3\cdot p^{k+1}X$-bounded.
\end{enumerate}

The topological group $X$, being $\w$-narrow, contains a countable subset $Z=\{z_k\}_{k\in\w}$ such that $X=ZW_k$ for all $k\in\w$.
For every $k\in\w$ let $G_k=\overline{p^kT_p(X)}$. Since the subgroup $T_p(X)$ is $p$-singular, for every $n\in\IN$ we have $\overline{nT_p(X)}=\overline{p^dT_p(X)}$ where $d\in\w$ is the largest number such that $p^d$ divides $n$. This observation, the choice of the sequence $(W_k)_{k\in\w}$ and Lemma~\ref{l:converge} guarantee that the sequences $(W_k)_{k\in\w}$ and $(G_k)_{k\in\w}$ satisfy the conditions (1),(2) of Lemma~\ref{l:seq} and (1)--(4) of  Lemma~\ref{l:Hausdorff}.

Applying Lemmas~\ref{l:seq}, construct a sequence $(x_n)_{n\in\w}$ satisfying the conditions (3),(4) of this lemma. Applying Lemma~\ref{l:Hausdorff}, we conclude that this sequence is $\diamondsuit$-Hausdorff and the set $\{x_m\}_{m\in\w}$ is totally bounded in the topological group $X_\diamond=(X,\tau_\diamond)$. Choose any accumulation point $x_\infty\in \bar X_\diamond$ of the sequence $(x_n)_{n\in\w}$ (which exists since the set $\{x_m\}_{m\in\w}$ is totally bounded in the complete topological group $\bar X_\diamond$).

\begin{claim} For every $n\in\IN$ the power $x_\infty^n\in\bar X_\diamond$ does not belong to $X_\diamond$.
\end{claim}

\begin{proof} To derive a contradiction, assume that $x^n_\infty\in X_\diamond$. Let $d\in\w$ be the largest number such that $p^d$ divides $n$. The $p$-singularity of the subgroup $T_p(X)$ and Claim~\ref{cl:modulo-p} guarantee that $$\overline{nG_0}=\overline{n T_p(X)}=\overline{p^d T_p(X)}\not\subset \overline{p^{d+1}T_p(X)}=G_{d+1}.$$

Since $X=ZW_{d+1}$, there exists a number $l\in\w$ such that $x^n_\infty \in z_lW_{d+1}$. Let $\e=\frac1{2^l}$.
Since the sequence $(x^n_m)_{m\in\w}$ accumulates at $x_\infty^n$ in $X_\diamond$, the neighborhood $z_l\diamondsuit_\e W_{d+1}\in\tau_\diamond$ of $x_\infty^n$ contains a point $x^n_\lambda$ with $\lambda>\max\{l,n\}$. Consequently, $x^n_\lambda\in z_lyW_{d+1}$ for some $y\in \diamondsuit_\e$.
Write $y$ as $y=x_1^{\e_1}\cdots x_m^{\e_m}$ for some $m\ge \lambda$ and some integer numbers $\e_1,\dots,\e_m$ such that $\sum_{i=1}^m\e_i=0$ and $\sum_{i=1}^m\frac{|\e_i|}{2^i}<\e=\frac1{2^l}$. The last inequality implies that $\e_i=0$ for all $i\le l$ and $|\e_i|<2^i$ for all $i\le m$.

It follows that $x^n_\lambda y^{-1}\in z_lW_{d+1}$ and $x^n_\lambda y^{-1}=x_1^{\delta_1}\cdots x_m^{\delta_n}$ where $\delta_\lambda=n-\e_\lambda$ and $\delta_i=-\e_i$ for all $i\in\{1,\dots,m\}\setminus\{\lambda\}$. It follows that $|\delta_\lambda|=|n-\e_\lambda|\le n+2^\lambda<2^{\lambda+1}$ and hence $|\delta_i|<2^{i+1}$ for all $i\le m$.

Taking into account that $\sum_{i=1}^m\delta_i=n$ is not divisible by $p^{d+1}$, we conclude that for some $j\le m$ the number $\delta_j$ is not divisible by $p^{d+1}$. We can assume that $j$ is the largest possible number with this property, i.e. $x_i^{\delta_i}\in p^{d+1}T_p(X)\subset G_{d+1}$ for all $i\in(j,m]$. It follows that $j>\max\{l\}$ (as $\delta_i=0$ for $i\le l$).
Let $c\in\w$ be the largest number such that $p^c$ divides $\delta_j$. Taking into account that $\delta_j$ is not divided by $p^{d+1}$, we conclude that $c\le d$ and $G_d\subset G_c$. By the $p$-singularity of $T_p(X)$, $G_{|\delta_j|}=\overline{\delta_j T_p(X)}=\overline{p^c T_p(X)}=G_c\not\subset G_{c+1}$.

 Then $z_lW_{d+1}\ni x^n_\lambda y^{-1}=x_1^{\delta_1}\cdots x_m^{\delta_m}\in x_1^{\delta_1}\cdots x_{j}^{\delta_{j}}G_{d+1}$ and hence $$x_j^{\delta_j}\in x_1^{-\delta_1}\cdots x_{j-1}^{-\delta_{j-1}}z_lW_{d+1}G_{d+1}\subset F_jW_{c+1}G_{c+1},$$
which contradicts the condition (3) of Lemma~\ref{l:seq}.
\end{proof}
\end{proof}

Next we treat the case of topological groups with hypocompact exponent whose $p$-components have compact exponent.

\begin{lemma}\label{l:P} Assume that an $\w$-narrow complete Abelian topological group $X$ has hypocompact exponent and for every prime number $p$ the $p$-component $T_p(X)$ has precompact exponent. If $X$ fails to have compact exponent, then $X$ contains a $\diamondsuit$-Hausdorff sequence $(x_n)_{n\in\w}$ which has an accumulation point $x_\infty\in\bar X_\diamond$ such that  $x_\infty^n\notin X_\diamond$ for all $n\in\IN$.
\end{lemma}

\begin{proof}
By Lemma~\ref{l:converge}, the subgroup $\overline{\w X}=\bigcap_{n\in\IN}\overline{nX}$ is compact and for every neighborhood $U\subset X$ of the unit there exists a number $n\in\IN$ such that $\overline{nX}\subset \overline{\w X}U$.

By our assumption, for every prime number $p$ the $p$-component $T_p(X)$ of $X$ has precompact exponent. Consequently, there exists $s_p\in\w$ such that $p^{s_p}T_p(X)$ is precompact. We can assume that the number $s_p$ is the smallest possible. If $s_p=0$ then the subgroup $T_p(X)$ is precompact. If $s_p>0$, then $p^{s_p}T_p(X)$ is precompact but $p^{s_p-1}T_p(X)$ is not.

\begin{claim}\label{cl:s-comp} The subgroup $S=\sum_{p\in\IP}p^{s_p}T_p(X)$ is precompact.
\end{claim}

\begin{proof}
Given a closed neighborhood $U\subset X$ of the unit, choose a neighborhood $W\subset X$ of the unit such that $W^2\subset U$. By the hypocompactness of the exponent of $X$, there exist $m\in\IN$ and a finite set $F_1\subset X$ such that $\overline{mX}\subset F_1W$. For every $p\notin \IP_m$ the $p$-locality of the subgroup $T_p(X)$ implies that $T_p(X)=\overline{mT_p(X)}\subset \overline{mX}$. Since $\overline{mX}$ is a subgroup of $X$, we also get $\sum_{p\in\IP\setminus \IP_m}T_p(X)\subset \overline{mX}\subset F_1W.$

The precompactness of the groups $p^{s_p}T_p(X)$ for $p\in\IP_m$ implies the precompactness of their (finite) sum $\sum_{p\in\IP_m}p^{s^p}T_p(X)$. So, we can find a finite subset $F_2\subset X$ such that $\sum_{p\in\IP_m}p^{s^p}T_p(X)\subset F_2W$. Then for the finite set $F=F_1F_2$ we have the desired inclusion
$$S=\sum_{p\in\IP}p^{s_p}T_p(X)=\sum_{p\in\IP_n}p^{s_p}T_p(X)\cdot \sum_{p\in\IP\setminus P_n}p^{s_p}T_p(X)\subset F_1W\cdot F_2W=FW^2\subset FU,$$
witnessing that the subgroup $S$ is precompact.
\end{proof}

\begin{claim} The set $\IP_{>}=\{p\in\IP:s_p>0\}$ is infinite.
\end{claim}

\begin{proof} To derive a contradiction, assume that the set $\IP_{>}$ is finite. In this case we shall prove that for the number $s=\prod_{p\in\IP_>}p^{s_p}$ the set $sX$ is precompact, which will contradict the assumption of Lemma~\ref{l:P}.

Indeed, for every $p\in\IP$ the $p$-singularity of the group $T_p(X)$ implies that $\overline{sT_p(X)}=\overline{p^{s_p}T_p(X)}$. By Lemmas~\ref{l:fg} and \ref{l:Tp-dense}, the subgroup $\sum_{p\in\IP}T_p(X)$ is dense in $X$. By the continuity of the map $X\to X$, $x\mapsto x^s$, the subgroup $s(\sum_{p\in\IP}T_p(X))=\sum_{p\in\IP}sT_p(X)$ is dense in $sX$. Then the closed subgroup $\overline{\sum_{p\in\IP}\overline{sT_p(X)}}=
\overline{\sum_{p\in\IP}\overline{p^{s_p}T_p(X)}}=\bar S$ coincides with $\overline{sX}$. By Claim~\ref{cl:s-comp}, the subgroup $\bar S=\overline{sX}$ is compact, witnessing that $X$ has precompact exponent. But this contradicts the assumption of Lemma~\ref{l:P}.
\end{proof}

The compactness of the subgroups $\overline{\w X}$ and $\bar S$ imply that the compactness of the subgroup $K=\overline{\w X}\cdot\bar S$. To shorten notations, for every prime number $p\in\IP_{>}$ denote the subgroup $p^{s_p-1}T_p(X)$ by $S_p(X)$ and observe that it is not precompact but $pS_p(X)=p^{s_p}T_p(X)$ is precompact.

\begin{claim}\label{cl:PW} There exists a strictly increasing sequence $(p_n)_{n\in\w}$ of prime numbers and a decreasing sequence $(W_n)_{n\in\w}$ of neighborhoods of the unit in $X$ satisfying the following conditions for every $n\in\w$:
\begin{enumerate}
\item $W_n=W_n^{-1}$ and $W_{n+1}^2\subset W_n$;
\item $s^{p_n}>0$;
%\item $W_p^2\subset W_q$ for any $q\in P$ with $q<p$;
\item $S_{p_n}\!(X)$ is not $KW_n^3$-bounded;
\item $\sum_{p\ge  p_{n+1}}T_p(X)\subset \overline{\w X}\cdot W_n\subset KW_n$.
\end{enumerate}
\end{claim}

\begin{proof} The construction of the sequences $(p_n)_{n\in\w}$ and $(W_n)_{n\in\w}$ is inductive. To start the inductive construction, let $p_0$ be the smallest prime number with $s^{p_0}>0$. Since the group $K$ is compact and the subgroup $S_{p_0}\!(X)$ is not precompact, it is not $KW^3_0$-bounded for some neighborhood $W_0\subset X$ of the unit in $X$.

Assume that for some $n\in\IN$ the prime number $p_{n-1}$ and a neighborhood $W_{n-1}$ have been constructed. By the hypocompactness of the exponent of $X$ and Lemma~\ref{l:converge}, there exists a number $m\in\IN$ such that $\overline{mX}\subset \overline{\w X}\cdot W_{n-1}$. Choose any prime number $p_{n}\in \{p\in\IP:s_p>0\}$ such that $p_n>\max(\IP_m\cup\{p_{n-1}\})$. Then for any prime number $p\ge p_n$ the number $m$ is not divided by $p$ and by the $p$-singularity of the group $T_p(X)$ its subgroup $mT_p(X)$ is dense in $T_p(X)$.
Then $T_p(X)=\overline{mT_p(X)}\subset \overline{mX}$. Since $\overline{mX}$ is a subgroup, we also get $\prod_{p\ge p_n}T_p(X)\subset\overline{mX}\subset \overline{\w X}\cdot W_{n-1}$.

Since the group $\overline{\w X}$ is compact and the group $S_{p_n}\!(X)$ is not precompact, it is not $\overline{\w X}W_n^3$-bounded for some neighborhood $W_n\subset X$ of the unit such that $W_n=W_n^{-1}$ and $W_n^2\subset W_{n-1}$.
This completes the inductive step.
\end{proof}

Since the topological group $X$ is $\w$-narrow, there exists a countable subset $Z=\{z_k\}_{k\in\w}$ such that $X=ZW_k$ for all $k\in\w$.
For every $k\in\w$ let $G_k$ be the closure of the subgroup $K\cdot\sum_{n=k}^\infty S_{p_n}\!(X)$ in $X$.

\begin{claim}\label{cl:div} For numbers $n\in\IN$ and $k\in\w$ the inclusion $nG_0\subset G_k$ holds if and only if $\{p_i:i<k\}\subset\IP_n$.
\end{claim}

\begin{proof} If $\{p_i:i<k\}\subset \IP_n$, then $$n\sum_{i\in\w}S_{p_i}\!(X)=\sum_{i\in\w}nS_{p_i}\!(X)\subset \sum_{i<k}p_iS_{p_i}\!(X)\cdot \sum_{i\ge k}S_{p_i}\!(K)\subset K\cdot\sum_{i\ge k}S_{p_i}\!(X)\subset G_k$$and hence $nG_0=n\overline{\sum_{i\in\w}S_{p_i}\!(X)}\subset \bar G_k=G_k$.

Next, assume that for some $i<k$ the number $p_i$ does not divide $n$.
The $p_i$-singularity of the $p_i$-component $T_{p_i}\!(X)$ implies that $\overline{nS_{p_i}\!(X)}=\overline{S_{p_i}\!(X)}$ and then $\overline{nG_0}$ contains the subgroup $\overline{S_{p_i}\!(X)}$ which is not $KW_{i}^3$-bounded by condition (3) of Claim~\ref{cl:PW}.
 On the other hand, by the condition (4) of Claim~\ref{cl:PW}, the group $G_k=\overline{\sum_{j\ge k}KS_{p_j}\!(X)}$ is contained in $KW_k\subset KW_i^3$ and hence is $KW_i^3$-bounded. So, $\overline{nG_0}\not\subset G_k$.
\end{proof}

Let us show that the sequences $(W_k)_{k\in\w}$ and $(G_k)_{k\in\w}$ satisfy the conditions (1)--(4) of Lemma~\ref{l:Hausdorff}. The condition (1) coincides with the condition (1) in Claim~\ref{cl:PW} and hence is satisfied. To check the condition (2), fix two numbers $n\in\IN$, $k\in\w$, and assume that $nG_0\not\subset G_k$. By Claim~\ref{cl:div}, $n$ is not divisible by some prime number $p_i$ with $i<k$. Then $\overline{S_{p_i}\!(X)}=\overline{nS_{p_i}\!(X))}
 \subset \overline{nG_0}$. By the condition (3) of Claim~\ref{cl:PW}, the set $S_{p_i}(X)$ is not $KW_i^3$-bounded and hence $\overline{nG_0}$ is not $KW_i^3$-bounded and $nG_0$ is not $KW_k^2$-bounded (as  $i<k$ and $\overline{W_i^3}\subset W_i^4\subset W_k^2$).

Since $K\subset G_\w$, the conditions (3), (4) of Lemma~\ref{l:Hausdorff} will follow as soon as we check that for every neighborhood $U\subset X$ of the unit there exists $k\in\w$ such that $G_k\subset KU$. By Lemma~\ref{l:converge}, there exists $n\in\IN$ such that $\overline{nX}\subset\overline{\w X}\cdot U$.
Choose $k\in\IN$ so large that for every $i\ge k$ the prime number $p_i$ does not divide $n$ and hence $S_{p_i}\!(X)\subset T_{p_i}\!(X)=\overline{nT_{p_i}\!(X)}\subset \overline{nX}$. Then $$
G_k=\overline{\sum_{i\ge k}KS_{p_i}\!(X)}\subset K\cdot \overline{nX}\subset K{\cdot}\overline{\w X}{\cdot}U=KU.$$

Therefore the conditions (1)--(4) of Lemma~\ref{l:Hausdorff} are satisfies. Then the conditions (1),(2) of Lemma~\ref{l:seq} are satisfied, too. So, we can apply Lemmas~\ref{l:seq} and \ref{l:Hausdorff} and construct a $\diamondsuit$-Hausdorff sequence $\{x_k\}_{k\in\w}\subset G_0$ satisfying the condition (3),(4) of Lemma~\ref{l:seq} and such that the set $\{x_k\}_{k\in\w}$ is totally bounded in the topological group $X_\diamond=(X,\tau_\diamond)$ and hence has an accumulation point $x_\infty$ in the completion $\bar X_\diamond$ of the topological group $X_\diamond$.

\begin{claim} For every $n\in\IN$ the power $x_\infty^n\notin X_\diamond$.
\end{claim}

\begin{proof} To derive a contradiction, assume that $x^n_\infty\in X_\diamond$. Let $d\in\w$ be the smallest number such that the prime number $p_d$ does not divide $n$. %We claim t The $p_\mu$-singularity of the subgroup $T_{p_\mu}(X)$ guarantee that $\tilde T_{p_\mu}(X)\subset \overline{nG_0}$ and hence $nG_0\notin=\overline{n T_p(X)}=\overline{p^d T_p(X)}\not\subset \overline{p^{d+1}T_p(X)}=G_{d+1}.$$

Since $X=ZW_{d+1}$, there exists a number $l\in\w$ such that $x^n_\infty \in z_lW_{d+1}$. Let $\e=\frac1{2^l}$.
Since the sequence $(x^n_m)_{m\in\w}$ accumulates at $x_\infty^n$ in $X_\diamond$, the neighborhood $z_l\diamondsuit_\e W_{d+1}\in\tau_\diamond$ of $x_\infty^n$ contains a point $x^n_\lambda$ with $\lambda>\max\{l,n\}$. Consequently, $x^n_\lambda\in z_lyW_{d+1}$ for some $y\in \diamondsuit_\e$.
Write $y$ as $y=x_1^{\e_1}\cdots x_m^{\e_m}$ for some $m\ge \lambda$ and some integer numbers $\e_1,\dots,\e_m$ such that $\sum_{i=1}^m\e_i=0$ and $\sum_{i=1}^m\frac{|\e_i|}{2^i}<\e=\frac1{2^l}$. The last inequality implies that $\e_i=0$ for all $i\le l$ and $|\e_i|<2^i$ for all $i\le m$.

It follows that $x^n_\lambda y^{-1}\in z_lW_{d+1}$ and $x^n_\lambda y^{-1}=x_1^{\delta_1}\cdots x_m^{\delta_n}$ where $\delta_\lambda=n-\e_\lambda$ and $\delta_i=-\e_i$ for all $i\in\{1,\dots,m\}\setminus\{\lambda\}$. Also,  $|\delta_\lambda|=|n-\e_\lambda|\le n+2^\lambda<2^{\lambda+1}$ and hence $|\delta_i|<2^{i+1}$ for all $i\le m$.

Let $\mu\in\w$ be the smallest number for which there exists a number $i\le m$ such that $\delta_i$ is not divisible by the prime number $p_\mu$. Such number exists and is $\le d$ since the number $n=\sum_{i=1}^m\delta_i$ is not divisible by $p_d$. The minimality of $\mu$ guarantees that for every $i\le m$ the number $\delta_i$ is divisible by the number $q=\prod_{j<\mu}p_j$ and hence $x_i^{\delta_i}\in qG_0\subset G_\mu$ according to Claim~\ref{cl:div}.

Let $j\le m$ be the largest number for which $\delta_j$ is not divisible by $p_\mu$.  The maximality of $j$ guarantees that $\delta_i$ is divisible by $p_\mu$ for all $i\in(j,m]$.
It follows that $j>l$ (as $\delta_i=0$ for $i\le l$). Then $z_lW_{d+1}\ni x^n_\lambda y^{-1}=x_1^{\delta_1}\cdots x_m^{\delta_m}\in x_1^{\delta_1}\cdots x_{j}^{\delta_{j}}G_{\mu}$ and hence $x_j^{\delta_j}\in x_1^{-\delta_1}\cdots x_{j-1}^{-\delta_{j-1}}z_lW_{d+1}G_{\mu}\subset F_jW_{\mu}G_{\mu}$,
which contradicts the condition (3) of Lemma~\ref{l:seq} as $\delta_jX\not\subset G_\mu$ (by Claim~\ref{cl:div}).
\end{proof}
\end{proof}

\section{Treating topological groups which do not have hypocompact exponent}

The following lemma was proved by Ravsky~\cite{Ravsky2003}. We shall give an alternative proof of this lemma deriving it from Key Lemmas~\ref{l:seq} and \ref{l:Hausdorff}.

\begin{lemma}\label{l:non-hypo} If an Abelian topological group $X$ is not of hypocompact exponent, then $X$ contains a $\diamondsuit$-Hausdorff sequence $(x_m)_{m\in\w}$, which has an accumulation point $x_\infty\in \bar X_\diamond$ such that $x^n_\infty\notin X_\diamond$ for all $n\in\IN$.
\end{lemma}

\begin{proof} Since $X$ is not of hypocompact exponent, there exists a neighborhood $W_0=W_0^{-1}$ of the unit such that for every $n\in\IN$ the set $nX$ is not $W_0^2$-bounded.

Choose a sequence $(W_k)_{k=1}^\infty$ of the unit such that $W_k^{-1}=W_k$ and $W_{k}^2\subset W_{k-1}$ for all $k\in\IN$.

\begin{claim} There exists a countable subgroup $Z=\{z_k\}_{k\in\w}$ of $X$ such that for every $n\in\IN$ the subgroup $nZ$ is not $W_0$-bounded.
\end{claim}

\begin{proof} Using Zorn Lemma, for every $n\in\IN$ choose a maximal subset $M_n\subset X$ such that for any distinct elements $x,y\in M_n$ we have $x^n\notin y^nW_0^2$. The maximality of $M_n$ implies that for every $x\in X$ there exists $y\in M_n$ such that $x^n\in y^nW_0^2$ or $y^n\in x^nW_0^2$ and hence $x^n\in y^nW_0^{-2}=y^nW_0^2$. Then $x^n\in y^nW_0^2\subset nM_n{\cdot}W_0^2$ and hence $nX\subset nM_n{\cdot}W_0^2$. Since the set $nX$ is not $W_0^2$-bounded, the set $M_n$ is infinite. So, we can choose a countable subgroup $Z=\{z_k\}_{k\in\w}\subset X$ such that for every $n\in\IN$ the intersection $Z\cap M_n$ is infinite.

We claim that for every $n\in\IN$ the subgroup $nZ$ is not $W_0$-bounded. Assuming the opposite, we could find a finite subset $F\subset X$ such that $nZ\subset FW_0$. Since the set $Z\cap M_n$ is infinite, for some $x\in F$ there are two distinct points $y,z\in Z\cap M_n$ such that $y^n,z^n\in xW_0$. Then $z^n\in xW_0\subset y^nW_0^{-1}W_0=y^nW_0^2$, which contradict the choice of the set $M_n$.
\end{proof}

Let $G_0=\bar Z$ and $G_k=\{e\}$ for all $k>0$.
It is clear that the sequences $(W_k)_{k\in\IN}$ and $(G_k)_{k\in\IN}$ satisfy the conditions (1),(2) of Lemma~\ref{l:seq} and (1)--(4) of Lemma~\ref{l:Hausdorff}. Applying these two lemmas, we obtain a $\diamondsuit$-Hausdorff sequence $(x_n)_{n\in\IN}$ satisfying the conditions (3),(4) of Lemma~\ref{l:seq} such that the set $\{x_n\}_{n\in\w}$ is totally bounded in the topological group $X_\diamond:=(X,\tau_\diamond)$ and hence has an accumulation point $x_\infty$ in the completion $\bar X_\diamond$ of $X_\diamond$.

\begin{claim} For every $n\in\IN$ the point $x^n_\infty$ does not belong to $X_\diamond$.
\end{claim}

\begin{proof}
Assuming that $x_\infty^n\in X_\diamond$, we would conclude that for every neighborhood $V\subset X$ of the unit the neighborhood $x_{\infty}^n\diamondsuit_1V\subset G_0V\in\tau_\diamond$ intersects the set $\{x^n_k\}_{k\in\w}\subset G_0$, which implies that $x^n_\infty\in G_0=\bar Z$. Then there exists a number $\lambda\in \w$ such that $x_\infty^n\in z_\lambda W_1$.

Let $\e=\frac1{2^\lambda}$.
Since the sequence $(x^n_m)_{m\in\w}$ accumulates at $x^n_\infty$ in the topological group $X_\diamond$, the neighborhood $z_l\diamondsuit_\e W_1\in\tau_\diamond$ of $x^n_\infty$ contains a point $x^n_k$ with $k>\max\{\lambda,n\}$. Consequently, $x^n_k\in z_\lambda y W_1$ for some $y\in \diamondsuit_\e$.
Write $y$ as $y=x_1^{\e_1}\cdots x_m^{\e_m}$ for some $m\ge k$ and some integer numbers $\e_1,\dots,\e_m$ such that $\sum_{i=1}^m\e_i=0$ and $\sum_{i=1}^m\frac{|\e_i|}{2^i}<\e$. The last inequality implies that $\e_i=0$ for all $i\le \lambda$.

It follows that $x^n_k y^{-1}\in z_\lambda W_1$ and $x^n_k y^{-1}=x_1^{\delta_1}\cdots x_m^{\delta_n}$ where $\delta_k=n-\e_k$ and $\delta_i=-\e_i$ for all $i\in\{1,\dots,m\}\setminus\{k\}$.
We claim that $|\delta_i|\le 2^{i+1}$ for all $i\le m$. For $i\ne k$ this follows from $|\delta_i|=|\e_i|<2^i$. For $i=k$, we have $|\delta_k|=|n-\e_k|\le n+|\e_k|\le k+2^k<2^{k+1}=2^{i+1}$.

Since $\sum_{i=1}^m\delta_i=n$, for some $i\le m$ the number $\delta_i$ is not equal zero. Let $j\le m$ be the largest number such that $\delta_j\ne 0$. It follows that $j>\lambda$ (as $\delta_i=0$ for $i\le \lambda$).
 Then $z_\lambda W_1\ni x^n_k y^{-1}=x_1^{\delta_1}\cdots x_m^{\delta_m}=x_1^{\delta_1}\cdots x_j^{\delta_j}$ and hence $x_j^{\delta_j}\in x_1^{-\delta_1}\cdots x_{j-1}^{-\delta_{j-1}}z_\lambda W_1\subset F_jW_1G_1$,
which contradicts the condition (3) of Lemma~\ref{l:seq} as $\max_{i\le j}|\delta_j|\le \max_{i\le j}2^{i+1}\le 2^{j+1}$ and $\delta_jX\not\subset G_1=\{e\}$.
\end{proof}
\end{proof}

\section{Joining pieces together}

In this section we combine Lemmas~\ref{l:p-local}, \ref{l:P}, \ref{l:non-hypo} and prove the ``unbounded'' part of Theorem~\ref{t:main}.
First observe that these three lemmas imply the following corollary.

\begin{corollary}\label{c:final} If an $\w$-narrow complete Abelian topological group $X$ does not have compact exponent, then $X$ admits a weaker Hausdorff group topology $\tau_\diamond$ such that the completion $\bar X_\diamond$ of the topological group $X_\diamond=(X,\tau_\diamond)$ contains an element $y\in\bar X_\diamond$ such that $y^n\notin X_\diamond$ for all $n\in\IN$.
\end{corollary}

Now we can prove the promised ``unbounded'' part of Theorem~\ref{t:main}.

\begin{theorem}\label{t:unbound} For a complete Abelian topological group $X$ the following conditions are equivalent:
\begin{enumerate}
\item $X$ is complete and has compact exponent;
\item for any continuous homomorphism $f:X\to Y$ to a powertopological semigroup $Y$ and every point $y\in \overline{f(X)}\subset Y$ there exists a number $k\in\IN$  such that $y^k\in f(X)$;
\item for any injective continuous homomorphism $f:X\to Y$ to a topological group $Y$ and every point $y\in \overline{f(X)}\subset Y$ there exists a number $k\in\IN$ such that $y^k\in f(X)$.
\end{enumerate}
\end{theorem}

\begin{proof} The implication $(1)\Ra(2)$ follows from Theorem~\ref{t:bound} and $(2)\Ra(3)$ is trivial. To prove that $(3)\Ra(1)$, assume that  the compact Abelian topological group $X$ does not have compact exponent. Then for every $n\in\IN$ the set $nX$ is not precompact. Using Lemma~\ref{l:separ}, we can find a closed separable subgroup $Z\subset X$ such that for every $n\in\IN$ the subgroup $nZ$ is not precompact, which means that $Z$ does not have compact exponent. By Corollary~\ref{c:final}, $Z$ admits a weaker Hausdorff group topology $\tau_\diamond$ such that the completion $\bar Z_\diamond$ of the topological group $Z_\diamond=(Z,\tau_\diamond)$ contains an element $z\in\bar Z_\diamond$ such that $z^n\notin Z_\diamond$ for all $n\in\IN$.

Let $\Tau_e$ be the family of all open neighborhoods of the unit in the topological group $X$. It is easy to se that the family
$$\tau_e=\{UV:U\in\Tau_e,\;e\in V\in\tau_\diamond\}$$ satisfies the Pontryagin Axioms \cite[1.3.12]{AT} and hence is a neighborhood base at the unit of some Hausdorff group topology $\tau$ in which the subgroup $Z$ remains closed and the subspace topology on $Z$ inherited from $(X,\tau)$ coincides with the topology $\tau_\diamond$. Then the completion $\bar X_\tau$ of the topological group $X_\tau:=(X,\tau)$ contains the completion $\bar Z_\diamond$ of the topological group $Z_\diamond$ and $\bar Z_\diamond\cap X_\tau=Z_\diamond$. Consequently, the element $z\in \bar Z_\diamond\subset \bar X_\tau$ has the required property: $z^n\notin X_\tau$ for all $n\in\IN$.
\end{proof}

\section*{Acknowledgements}

The author expresses his thanks to Alex Ravsky for fruitful discussions on the topic of this paper and to Michael Megrelishvili for helpful  remarks and comments on the text. Special thanks are due to Dikran Dikranjan who told to the author the genuine story of the proof of Prodanov-Stoyanov Theorem from \cite{DPS} (which differs a bit from the original proof of this theorem in \cite{PS}) and suggested to dedicate this paper to the memory of Ivan Prodanov who died of heart attack at the age of 49 --- precisely the age of the author at the moment of writing this paper.
%\newpage

\end{document}